\documentclass[12pt,reqno]{amsart}

\usepackage{a4wide}
\usepackage{amsmath}
\usepackage{amssymb}
\usepackage{mathrsfs}
\usepackage{amsthm}
\usepackage{comment}
\usepackage{color}
\usepackage{enumitem}
\usepackage{euscript}
\usepackage{comment}

\newlist{mylist}{enumerate}{2}
\setlist[mylist,1]{label=(\alph*),font=\normalfont}
\setlist[mylist,2]{label=(\roman*),font=\normalfont}

\title{On ideals of product of commutative rings and their applications}
\author{A. R. Aliabad}
\address{Department of Mathematics, Shahid Chamran University, Ahvaz, Iran}
\email{aliabady\_r@scu.ac.ir}

\author{M. Badie}
\address{Department of Basic Sciences, Jundi-Shapur University of Technology, Dezful, Iran}
\email{badie@jsu.ac.ir}

\author{F. Obeidavi}
\address{Department of Mathematics, Shahid Chamran University, Ahvaz, Iran}
\email{foadobeidavinazem@gmail.com}


\theoremstyle{plain}
\newtheorem{theorem}{Theorem}[section]
\newtheorem{lemma}[theorem]{Lemma}

\newtheorem{proposition}[theorem]{Proposition}
\newtheorem{corollary}[theorem]{Corollary}
\theoremstyle{definition}
\newtheorem{example}[theorem]{Example}

\newcommand{\M}{\mathcal{M}}
\newcommand{\N}{\mathbb{N}}
\newcommand{\R}{\mathbb{R}}
\newcommand{\Z}{\mathbb{Z}}

\newcommand{\I}{\mathfrak{i}}

\newcommand{\Ge}[1]{\big< #1 \big>}

\newcommand{\ff}{if and only if }
\newcommand{\RTA}{\Rightarrow}
\newcommand{\LTA}{\Leftarrow}
\newcommand{\LRTA}{\Leftrightarrow}

\newcommand{\Mi}{\mathrm{Min}}
\newcommand{\Ma}{\mathrm{Max}}
\newcommand{\An}{\mathrm{Ann}}
\newcommand{\Sp}{\mathrm{Spec}}

\newcommand{\Jac}{\mathrm{Jac}}
\newcommand{\In}{\mathrm{int}}
\newcommand{\Cl}{\mathrm{cl}}
\newcommand{\e}[1]{e_{_{#1}}}
\newcommand{\AI}{I}
\newcommand{\BI}{J}
\newcommand{\MI}{M}
\newcommand{\NI}{N}
\newcommand{\PI}{P}

\newcommand{\sub}{\subseteq}

\newcommand{\emp}{\emptyset}
\newcommand{\all}{\forall}
\newcommand{\ex}{\exists}
\newcommand{\la}{\lambda}
\newcommand{\LA}{\Lambda}

\newcommand{\we}{\wedge}
\newcommand{\bve}{\bigvee}
\newcommand{\bwe}{\bigwedge}
\newcommand{\bcap}{\bigcap}
\newcommand{\bcup}{\bigcup}

\newcommand{\CF}{\mathcal{F}}
\newcommand{\CG}{\mathcal{G}}
\newcommand{\CI}{\mathcal{I}}

\newcommand{\CU}{\mathcal{U}}

\newcommand{\FF}{\mathfrak{F}}

\newcommand{\BBN}{\mathbb{N}}

\newcommand{\BBR}{\mathbb{R}}

\begin{document}
	\begin{abstract}
			In this paper, leveraging the recent achievements of researchers, we have revisited the family of ideals of product of commutative rings. We demonstrate that if $ \{ R_\alpha \}_{\alpha \in A} $ is an infinite family of rings, then $ \left| \Ma \left( \prod_{\alpha \in A} R_\alpha \right) \right| \geqslant 2^{2^{|A|}} $. Notably, if these rings are local then the equality holds. We establish that $ \Ma(R_\alpha) $ is homeomorphic to a closed subset of $ \Ma \left( \prod_{\alpha \in A} R_\alpha \right) $, for each $ \alpha \in A $. Additionally, we show that $ \Ma(R) $ is disconnected \ff $ R $ is direct summand of its two proper ideals. We deduce that if the intersection of each infinite family of maximal ideals of a ring is zero, then the ring is not direct summand of its two proper ideals. Furthermore, we prove that for each ring $R$, $ C\left(\Ma(R)\right) $ is isomorphic to $ C\left(\Ma\left(C(Y)\right)\right) $, for some compact $T_4$ space $Y$. Finally, we explore that $h_M(x)$'s can define roles of zero-sets.
	\end{abstract}
	
	\maketitle

\section{Introduction}
	
The ideals, particularly prime, minimal prime and maximal ideals of rings, play crucial roles in classifying certain classes of rings. Therefore, studying them is a significant focus in algebra literature. One important class of rings is the product of rings. Characterizing the ideals, prime ideals, minimal prime ideals and maximal ideals of products of rings by the algebraic properties of component rings is a  goal for researchers. For some examples see \cite{Anderson2008Ideals,Finocchiaro2023Prime,Gilmer1992Product,Gilmer1989imbedding,Gilmer1995Infinite,Levy1991Prime,Olberding2005Prime,Tarizadeh2020Projectivity,Tarizadeh2023Tame}.

Using the concept of filters on different families and methods, the ideals of products of rings are characterized. For some examples see \cite{Aliabad2018Bourbaki,Finocchiaro2023Prime,Levy1991Prime,Tarizadeh2023Tame}. Additionally, by employing the noted filter method, the cardinality of the space of minimal prime ideals of some rings and infinite products of rings is characterized. For some examples see \cite{Aliabad2018Bourbaki,Levy1991Prime,Tarizadeh2023Tame}. 

From \cite[Proposition 1]{Levy1991Prime}, we can easily conclude that if $ A $ is infinite then $ \big| \Mi\big( \prod_{\alpha \in A} \Z \big) \big| = 2^{2^{|A|}} $.  This statement is generalized in \cite[Corollary 4.15]{Aliabad2018Bourbaki}. It is shown that if $ \{ D_\alpha \}_{\alpha \in A} $ is an infinite family of integral domains then the cardinal of $ \mathrm{Min} \left( \prod_{\alpha \in A} D_\alpha \right) $ is $ 2^{2^{|A|}} $. This statement is further explained and shown in \cite[Theorem 2.15]{Tarizadeh2023Tame}.

Moreover, the space of prime ideals of commutative rings and their subspaces, such as the space of minimal prime ideals and maximal ideals, are studied to characterize algebraic concepts using topological concepts. For some examples, see \cite{Aliabad2013Fixedplace,Aliabad2018Bourbaki,Aliabad2021Commutative,Gillman1957Rings,Henriksen1965Space,Kohls1958space}. 

$ \Sp(C(X)) $ equipped with the Zariski topology and its subspaces $ \Ma(C(X)) $ and $ \Mi(C(X)) $ are studied in several articles, for example see \cite{Aliabad2013Fixedplace,Aliabad2018Bourbaki,Dow1988Space,Henriksen1965Space,Walker1974Stone}. The most important conclusion in this literature is that $ \Ma(C(X)) $ is Stone-\v Cech compactification of $ X $.

The family of zero-sets of a space plays an important role in studying the relationship between the algebraic properties of the ring of real-valued continuous functions and the topological properties of the space. Consequently, it has become the foundation for introducing many concepts in the literature on rings of functions.

In continuation of researchers efforts to extend the concepts of rings of continuous functions to those of commutative rings, the authors in \cite{Aliabad2020Extension,Aliabad2021Commutative} have attempted to demonstrate that $ h_M(x) $'s can assume some roles of zero-sets. 

In this article, leveraging the recent achievements of researchers, we have revisited the concepts of:
\begin{mylist}
	\item The family of ideals of product of commutative rings.
	\item Particularly, the family of maximal ideals of product of commutative rings and its cardinality,
	\item The space of maximal ideals equipped by Zariski topology, 
	\item The ring of functions of the space of maximal ideals and 
	\item The role of $ h_M(x) $'s in the studying of algebraic concepts.
\end{mylist}

In Section 2, we recall some pertinent definitions and statements. In Section 3, we introduce a method to define an ideal using an ideal and a filter, where the ideal is an ideal of the product of a family of commutative rings and the filter is a filter on the index set of this family. Additionally, we present another method to define a filter using an ideal of the product of a family of commutative rings. In the remainder of this section, we explore the relationship between the aforementioned filters and ideals. We demonstrate that if $\{ R_\alpha \}_{\alpha \in A} $ is an infinite family of commutative rings, then $ \left| \Ma\left( \prod_{\alpha \in A} R_\alpha \right) \right| \geqslant 2^{2^{|A|}} $. Particularly, if these commutative rings are local rings, then the equality holds. Section 4 is devoted to study the space of $\Ma(R)$ equipped with Zariski topology in three parts. 

\subsection{The topological properties of $\Ma(R)$} 
Closure, interior, limit and isolated points of the space is characterized. It is proved that $ \Ma(R_\alpha) $ is homeomorphism to a closed subset of $ \Ma \left( \prod_{\alpha \in A} R_\alpha \right) $. Furthermore, it is deduced that if this family is finite, then $ \Ma \left( \prod_{i = 1}^n R_i \right)  $ is homeomorphism by the disjoint union of $ \Ma( R_i ) $'s. It is characterized that under certain conditions, $ \Ma(R) $ is discrete, cofinite or disconnected. In the final of this part as an application is deduced that if the intersection of each family of maximal ideals of a ring is zero, then this ring is not direct summand of two of its proper ideals. 

\subsection{The rings of real-valued continuous functions on $\Ma(R)$}
In this part, we prove that $ C\left(\Ma\big( \prod_{i=1}^n R_i \big)\right) $ is isomorphic to product of $ C\left(\Ma(R_i)\right) $'s; if $D$ is an integral domain, then $C\left(\Ma(D)\right)$ is isomorphic to the filed $\R$; and the ring of functions on the family of maximal ideals of any ring is isomorphic to the ring of functions on the family of maximal ideals of a ring of functions. 

\subsection{The role of $h_M(x)$'s}
In the final part, we recall and present statements demonstrating that $h_M(x)$'s can serve as zero-sets. Additionally, we characterize regular rings and rings without regular proper ideals using $h_M(x)$'s.

\section{Preliminary}

Let $X$ be a set and $ \leqslant $ a relation on $X$. We say $(X,\leqslant)$ is a directed set if $\leqslant$ is transitive and reflexive on $X$ and each finite subset of $X$ has an upper bound in $X$. A lattice $L$ in which every subset has a supremum is said to be a complete lattice. A complete lattice $L$ is said to be a frame if for every $a\in L$ and every $B\sub L$, we have $a\we{\bve}B=\bve_{b\in B}a\we b$. A complete lattice is called a symmetric frame if for every $a\in L$ and every $B\sub L$, we have:
\[ a\we{\bve}B=\bve_{b\in B}a\we b\quad,\quad a\vee{\bwe}B=\bwe_{b\in B}a\vee b.\]

Recall that, supposing $X$ is a partially ordered set, a function $f:X\to X$ is called a closure mapping if $f$ is increasing, idempotent, and extensive. A closure mapping $f$ on a lattice ordered is said to be nucleus if $f(a\we b)=f(a)\we f(b)$ for every $a,b\in X$. A frame homomorphism is a homomorphism of partially ordered set that preserves finite meets and arbitrary joins.

Let $X$ be a nonempty set. An intersection structure (briefly, $\cap$-structure or cap-structure) on $X$ is a nonempty family  $\mathcal{M}_X$  of subsets of $X$ which is  closed under arbitrary intersection.  In this case we say $(X,\mathcal{M}_X)$  is a cap-structure space. Clearly, if $(X,\mathcal{M}_X)$ is a cap-structure space, then $\mathcal{M}_X$ is a complete lattice in which  for each subfamily $\{A_i\}_{i\in I}$ of $\mathcal{M}_X$:
\[ \bigwedge _{i\in I}A_i=\bigcap_{i\in I}A_i \quad, \quad \bigvee _{i\in I}A_i=\bigcap\{B\in \mathcal{M}_X: \bigcup_{i\in I}A_i\subseteq B\}.\]

\begin{lemma}\label{distributive is frame}
	Let $L$ be a cap-structure in which for every upper directed subset $A$ of $L$, we have ${\bcup}A={\bve}A$. Then the following statements are equivalent.
	\begin{mylist}
		\item $L$ is a distributive lattice. \label{distributive}
		\item $L$ is a frame. \label{frame}
	\end{mylist}	
\end{lemma}
\begin{proof}
	\ref{distributive} $\RTA$ \ref{frame}. Suppose that $x\in L$ and $B\sub L$. Let $A$ be set of finite joins of elements of $B$. Clearly, $A$ is a upper directed subset of $L$ and so ${\bcup}A={\bve}A={\bve}B$. Thus, by hypothesis, we can write:
	$$x\we{\bve}B=x\cap{\bcup}A=\bcup_{a\in A}x\cap a=\bcup_{a\in A}x\cap\bve_{b\in F_a}b=\bcup_{a\in A}\bve_{b\in F_a}x\cap b=\bve_{b\in B}x\cap b.$$
	
	\ref{frame} $\RTA$ \ref{distributive}. It is clear.
\end{proof}

Everywhere in the article, it is assumed that all rings are commutative with identity. An $f$-ring is a subring of a product of totally-ordered rings that is also closed under the natural lattice operations. For each ring $R$, $\Sp(R)$, $\Mi(R)$, $\Ma(R)$ and $\Jac(R)$ represent the family of all prime ideals of $R$, the family of all minimal prime ideals of $R$, the family of all maximal ideals of $R$, and the intersection of all maximal ideals of $R$, respectively. We say an ideal $P$ is pseudoprime if $ xy=0 $ implies that either $x \in P$ or $y \in P$. We describe a ring as local if it has a unique maximal ideal. A ring is termed Gelfand if every prime ideal is contained in a unique maximal ideal. An ideal is called Hilbert if it is an intersection of maximal ideals. An element of a ring labeled as regular if it is not a zero-divisor. An ideal containing a regular element is referred to as a regular ideal. A ring $R$ is called von Neumann regular ring, if for every $x$ in $R$, there exits $y$ in $R$ such that $x=x^2y$. We say a ring is an integral domain if it has no zero-divisor element. For any subsets $S$ and $T$ of a ring $R$, we denote $\{x\in R:~x S\sub T\}$ by $(T:S)$. Specially, $(\{0\}:S)$ and $(\{0\}: x)$ are denoted by $\An(S)$ and $\An(x)$, respectively. Two ideals $I$ and $J$ of a ring $R$ are termed comaximal whenever $I+J=R$. An ideal is defined as semiprime if it is an intersection of prime ideals. If the zero ideal of a ring is semiprime, we say the ring is reduced. We say $R$ is semiprimitive, if $\Jac(R)$ is zero. A prime ideal $\PI_0$ in $\Mi(R)$ is called a Bourbaki associated prime divisor of $R$, if $\sqrt{\{0\}}\neq\bigcap_{\PI_0\neq\PI\in\Mi(R)}\PI$. The family of all Bourbaki associated prime divisor of $R$ is indicated by $\mathcal{B}(R)$. A family $\mathcal{A}$ of prime ideals is described as a fixed-place family if no element of this family can be omitted from $\bigcap\mathcal{A}$. For more information see \cite{Aliabad2013Fixedplace,Aliabad2018Bourbaki}. For each subset $S$ of $R$, the family $ \{\MI \in \Ma(R):~S \sub \MI \} $ is denoted by $h_M(S)$. Particularly, $h_M(\{x\})$ is denoted by $h_M(x)$. The collection $\{h_M(S):~S\sub M\}$ forms the family of closed subsets of a topology on $\Ma(R)$ called Zariski topology or hall-kernel topology. The intersection of any family $ \mathcal{A}$ of sets symbolized by $k(\mathcal{A})$. An ideal $I$ is termed (strong) $z$-ideal, if for each (finite subset $F$) element $x$ in $I$, ($kh_M(F)$) $kh_M(x)$ is contained in $I$. A ring $R$ is called arithmetical, if for each ideals $I$, $J$ and $K$ of $R$ we have $ I \cap (J+K) = (I \cap J) + (I \cap K) $. This is equivalent to say that for every ideals $I$, $J$ and $K$ of $R$ we have $ I + (J \cap K) = (I + J) \cap (I + K) $.

For each set $X$, $\tau=\{G\sub X:~G^c ~\text{is finite}\}\cup\{\emp\}$ is called cofinite topology. $\bigoplus_{\la \in\LA} X_\la $ is symbolized for the disjoint union of a family $ \{ X_\la \}_{\la \in\LA} $ of topological spaces. For each topological space $X$, $C(X)$ denotes the family of all real-valued continuous functions on $X$. A space is termed $P$-space if each zero-set is open and is called almost $P$-space if each non-empty zero-set has a non-empty interior. In this article, a filter may be not proper. A maximal element of the family of all proper filters is called ultrafilter. We say a filter $\CF$ is fixed whenever $ \bigcap \CF\neq\emp$.

The reader is referred to \cite{Atiyah1969Introduction,Gillman1960Rings,Sharp2000Steps,Willard1970General,Blyth2005Lattices,Davey2002Introduction,Gratzer2011Lattice,Dikranjan2013Categorical,Abdollahpour2024Homorphisms,Hashemi2020Cap} for undefined terms and notations. 

\section{New ideals in $R=\prod_{\la\in\LA}R_\la$}

Henceforth, in this section $ R $ denotes $ \prod_{\la \in\LA} R_\la $, in which $ \{ R_\la \}_{\la \in\LA} $ is a family of rings. For every $ x \in R $, the set $ \{ \la \in\LA : x_\la=0 \} $ is denoted by $Z(x)$ and $ A \setminus Z(x) $ is denoted by $Coz(x)$. The set $\{e \in  I : ~\all\la\in\LA,~e_\la=0 \quad\text{ or } \quad e_\la=1\}$ is denoted by $E(R)$. Also, for every ideal $I$ of $R$, we denote by $E(I)$ the set $I\cap E(R)$. It is clear, for every $Z\sub  A $, there is a unique element $e_{_Z} $ in $ E(R)$ such that $ Z(e_{_Z})=Z$. Clearly, for each pair $ e , e' \in E(R) $
\begin{equation}
	e=e' \quad \text{ \ff } \quad Z(e)=Z(e') \label{equality in E(R)}
\end{equation}
and 
\begin{equation}
	e+e' \in E(R) \quad \text{ \ff } \quad Z(e) \cup Z(e')=X \label{sum of E(R)'s in E(R)} 
\end{equation}

\begin{lemma}\label{Z and e elementary propoeties}
	Suppose that $x , y \in R$, $\AI$ is an ideal of $R$ and $Y , Z \sub\LA$. Then
	
	\begin{mylist}
		\item $Z(x) \cap Z(y) \sub Z(x+y)$. If $x , y \in E(R)$, then the equality holds. \label{Z cap sub}
		\item $Z(x) \cap Z(y)=Z(x+e_{_{Coz(x)}} y )=Z(x e_{_{Coz(y)}}+y )$. \label{Z cap =}
		\item If $x , y \in E(\AI)$, then $x+e_{Coz(x)} y , x e_{Coz(y)}+y \in E(\AI)$. \label{x + e_Coz in E(I)}
		\item $ Z(x) \cup Z(y) \sub Z(xy) $. If either $ x , y \in E(R) $ or $ R_\la $ is an integral domain, for every $ \la \in\LA $, then the equality holds. \label{Z cup}
		\item $Z(\e{Y}+\e{Z})=Z(e_{_Y})\cap Z(e_{_Z})=Y\cap Z$. \label{Z(sum)}
		\item $Z(e_{_Y}e_{_Z})=Z(e_{_Y})\cup Z(e_{_Z})=Y\cup Z$. \label{Z(multiple)}
		\item $e_{_{Y\cup Z}}=e_{_Y}e_{_Z}$. \label{e cup}
		\item $e_{_{Y \cap Z}} = e_{_{Y \cup Z^c}} + e_{_Z} = e_{_{Z^c}}e_{_Y}+e_{_Z}$, so $e_{_Z} + e_{_{Z^c}} = 1$. \label{e cap}
	\end{mylist}
\end{lemma}
\begin{proof}
	\ref{Z cap sub}, \ref{x + e_Coz in E(I)} and \ref{Z cup}. They are straightforward.

	\ref{Z cap =}. Clearly, we have:
	\begin{align*}
		Z(x)\cap Z(y) & = Z(x)\cap(Coz(x)\cup Z(y))=Z(x)\cap(Z(e_{Coz(x)})\cup Z(y)) \\
					  & =Z(x)\cap Z(e_{Coz(x)}y)=Z(x+e_{Coz(x)}y).
	\end{align*}
	Similarly, we can show that $ Z(x) \cap Z(y)=Z(e_{_{Coz(y)}} x+y) $.
	
	\ref{Z(sum)} and \ref{Z(multiple)}. They follow from parts \ref{Z cap sub} and \ref{Z cup}.
	
	\ref{e cup}. By the fact \eqref{equality in E(R)}, it follows from part \ref{Z(multiple)}.
	
	\ref{e cap}. By part \ref{Z(multiple)}, $ Z(e_{Y \cup Z^c}+e_{_Z})=Z(e_{Y \cup Z^c}) \cap Z(e_{_Z})=\big( Y \cup Z^c\big)\cap Z=Y \cap Z $. By the fact \eqref{sum of E(R)'s in E(R)}, $ e_{_{Y \cup Z^c}}+e_{_Z} \in E(R) $ and therefore $ e_{_{Y \cap Z}}=e_{_{Y \cup Z^c}}+e_{_Z} $, by the fact \eqref{equality in E(R)}. Hence $ e_{_Z}+e_{_{Z^c}}=e_{_Z}+e_{_{\emp \cup Z^c}}=e_{_{\emp \cap Z}}=e_{_\emp}=1  $.
\end{proof}

By $\I(\CF,I)$ we mean the set $\{x\in R:~\ex F\in\CF\; x\e{F^c}\in\AI\}$, for each filter $\CF$ on $\LA$ and ideal $I$ of $R$.

\begin{lemma}\label{I(F,I) properties}
	Suppose that $\CF$ is a filter on $\LA$, $I$ is an ideal of $R$ and $x \in R$. Then the following statements hold:
	\begin{mylist}
		\item $e_{_F}\in\I(\CF,I)$ for every $F\in\CF$. \label{e_F in I}
		\item $x \in \I(\CF,I)$ \ff there exists $y$ in $I$ such that $Z(x-y) \in \CF$. \label{x in I}
		\item $\I(\CF,I)$ is an ideal containing $I$. \label{I is an ideal}
		\item $\I(\CF,\{0\}) = \Ge{\{ e_{_Z} : Z \in \CF \}} = \{ x \in R : Z(x) \in \CF\} $. \label{I(F,0)}
		\item $\I(\CF,I)$ is a proper ideal \ff $\e{F^c}\notin I$ for every $F\in\CF$. \label{I is proper}
		\item Supposing $I_\la=\pi_\la(I)$ for every $\la\in\LA$, we have $\I(\CF,I)\sub\I(\CF,\prod_{\la\in\LA}I_\la)$ and the converse is true \ff the following implication holds: \label{I is subset of product}
		\[ \ex F\in\CF \quad \all \la\in F,\quad x_\la\in I_\la \quad \RTA \quad \ex F\in\CF \quad xe_{_F^c}\in I \]
		\item Suppose that for each $\la \in \Lambda$, $\AI_\la$ is an ideal of $R_\la$. Then $x \in \I(\CF,\prod_{\la \in\LA} \AI_\la)$ if and only if $F=\{\la\in\LA:~x_\la \in \AI_\la\}\in\CF$.\label{x in I(product)}
		\item If $I$ is a prime ideal, $\I(\CF,\AI)\sub\I(\CG,\AI)\neq R$ and $\CG$ is an ultrafilter, then $\CF\sub\CG$. \label{F sub G}
		\item For every ideal $I$ of $R$, $\I(\CF,I)=I$ \ff $\CF=\{X\}$. \label{I(I) = I}		
	\end{mylist}
\end{lemma}
\begin{proof}
	\ref{e_F in I}. Obviously, we can write:
	\[\all F\in\CF, \quad e_{_F}e_{_{F^c}}=0\in I \quad \Rightarrow \quad \all F\in\CF, \quad e_{_F}\in\I(\CF,I).\]
	
	\ref{x in I}$\RTA$. By the assumption there exists $F \in \CF$ such that $ y = x e_{_{F^c}} \in I $. Then for each $ \lambda \in F $ we have 
	\[ y_\lambda = x_\lambda (e_{_{F^c}}) \quad \Rightarrow \quad y_\lambda - x_\lambda = 0 \quad \Rightarrow \quad \lambda \in Z(x-y)\] 
	thus $ F \subseteq Z(x-y) $ and therefore $Z(x-y) \in \CF $.
	
	 \ref{x in I}$\LTA$. It is easy to see that $ x e_{_{Z(x-y)^c}} = y \in I $ and thus, by the assumption, $ x \in \I(\CF,I)$.
	 
	 \ref{I is an ideal}. Since $ 0\e{X^c} = 0 \in \AI $, it follows that $ 0 \in \I(\mathcal{F},\AI) \neq \emptyset $. Suppose that $ x , y \in \I(\mathcal{F},\AI) $ and $ r \in R $, then $ E , F \in \mathcal{F} $ exist such that $ x\e{E^c} , y\e{F^c} \in \AI $. So, by Lemma \ref{Z and e elementary propoeties}\ref{e cup}, 
	 \[ ( x + r y ) \e{( E \cap F )^c} = x \e{E^c} \e{F^c} + r y \e{F^c} \e{E^c} \in \AI \]
	 Since $ E \cap F \in \mathcal{F} $, it follows that $ x  + r y \in \I(\mathcal{F},\AI) $. Consequently, $ \I(\mathcal{F},\AI) $ is an ideal.
	 
	 \ref{I is proper}. It is clear that $1\notin\I(\CF, I)$ \ff $e_{_{F^c}}=1e_{_{F^c}}\notin I$ for every $F\in\CF$.
	 
	 \ref{I is subset of product} Since $I \subseteq \prod_{\alpha \in \LA} I_\la$, it follows that $\I(\CF,I)\sub\I(\CF,\prod_{\la\in\LA}I_\la)$. 
	 
	 We can write:
	 \begin{align*}
	 	\I(\CF,I) \subseteq \I(\CF,\prod_{\la \in \LA} I_\lambda) \quad
	 			& \equiv \quad x \in \I(\CF,I) \quad \RTA \quad x \in \I(\CF,\prod_{\la \in \LA} I_\lambda) \\
	 			& \equiv \quad \ex F\in\CF \quad xe_{_{F^c}}\in I \quad \RTA \quad \ex F\in\CF \quad \all \la\in F,\quad x_\la\in I_\la.	 	
	 \end{align*}
	 
	 \ref{x in I(product)} $\RTA$. Since $x\in \I(\CF,\prod_{\la\in\LA} \AI_\la)$, so $E \in \CF$ exists such that $x \e{E^c} \in\prod_{\la\in\LA}\AI$ and thus $x_\la\in \AI_\la$, for each $\la \in E$. This shows that $F \supseteq E \in \CF$ and therefore $F \in \CF$. 
	 
	 \ref{x in I(product)} $\LTA$. Clearly, $x \e{F^c} \in \prod_{\la\in\LA} \AI_\la$ and therefore $x \in \I( \CF, \prod_{\la \in\LA}\AI_\la)$.

	 \ref{F sub G}. On contrary, suppose that $ F \in \CF \setminus \CG $ exists, so $F^c \in \CG$. Since $ 1 \notin \I(\CG,I) \neq R $, it follows that $e_{_F} = 1 e_{_F} \notin I$ and, by the facts $ e_{_F} e_{_{F^c}} = 0 \in I $ and $I$ is prime, we have $ e_{_{F^c}} \in I$. Since $F \in \CF$ and $ 1 e_{_{F^c}} = e_{_{F^c}} $, it follows that $ 1 \in \I(\CF,I) \neq R $, which is a contradiction.   
	 
	 \ref{I(I) = I} $\RTA$. Suppose that $\CF=\{X\}$ and $x\in\I(\CF,I)$. By hypothesis, $x=x\e{\emptyset}=x\e{X^c}\in I$.
	 
	 \ref{I(I) = I} $\LTA$. Assume that $\CF\neq\{X\}$ and $I$ is a proper ideal of $R$ such that $\e{F^c}\in I$ for some $X\neq F\in\CF$. Thus, $1e_{_{F^c}}=e_{_{F^c}}\in I$ and so $1\in\I(\CF, I)$. Therefore, $I\neq R=\I(\CF, I)$.
	 
\end{proof}

Assume that $\FF$ is the family of all filters on $\LA$ (not necessarily proper), and $\CI(R)$ the family of all ideals of the ring $R$. Clearly, $\FF$ and $\CI(R)$ are complete lattices in which
\[ \bwe_{t\in T}I_t=\bcap_{t\in T}I_t \quad,\quad \bve_{t\in T}I_t=\sum_{t\in T}I_t,\]
\[\bwe_{t\in T}\CF_t=\bcap_{t\in T}\CF_t \quad,\quad \bve_{t\in T}\CF_t=\left\{\bcap_{d\in D}A_d:~D~\text{is a finite subset of }T, \quad A_d\in\CF_d\right\}.\]
Suppose that $I\in\CI(R)$ and $\CF\in\FF$. Define mappings $I:\FF\to\CI(R)$ with $I(\CF)=\I(\CF,I)$ and $ \CF:\CI(R)\to\CI(R)$ with $\CF(I)=\I(\CF,I)$. 

\begin{proposition}\label{F properites}
	Considering a filter $\CF$ on $\LA$ as map on $\CI(R)$ with $\CF(I)=\I(\CF,I)$, the following statements hold:
	\begin{mylist}
		\item $\CF$ is well-defined. \label{F is well}
		\item $\CF$ is an increasing map, i.e., if $I\sub J$, then $\I(\CF,\AI) \sub \I(\CF,\BI)$. \label{F is increasing}
		\item $\CF$ is extensive; i.e., $\CF(I)$ is an ideal containing $I$. \label{F is extensive}
		\item $\CF$ is idempotent; i.e., $\CF(\CF(I))=\CF(I)$ for every $I\in\CI(R)$. \label{F is idempotent}
		\item $\CF$ is a lattice homomorphism. \label{F is lattice homomorphism}
		\item $\CF$ preserves arbitrary joins. \label{F preserves joins}
		\item $\CF$ is nucleus. \label{F is nuclues}
		\item If $R$ is an arithmetical ring (i.e., $R_\la$'s are so), then $\CI(R)$ is a frame and consequently $\CF$ is a frame homomorphism. \label{F is frame homomorphism}
	\end{mylist}
\end{proposition}
\begin{proof}
	\ref{F is well} and \ref{F is extensive}. They follow from Lemma \ref{I(F,I) properties}\ref{I is an ideal}.
	
	\ref{F is increasing}. It is clear.
	
	\ref{F is idempotent}. Suppose that $J = \I(\CF,I)$. By part \ref{F is increasing}, we have $ J \sub \CF(J)$. Now suppose that $x \in \CF(J)$, then $F \in \CF$ exists such that $xe_{_{F^c}} \in J =\CF(I)$. It follows that there is $E \in \CF$ such that $xe_{_{F^c}}e_{_{E^c}} \in I$. This shows that $xe_{_{(F \cap E)^c}} \in I$. Since $F \cap E \in \CF$, it follows that $x \in J=\CF(I)$, by Lemma \ref{Z and e elementary propoeties}\ref{e cup}. Consequently, $J = \CF(J)$. Thus $\CF$ is idempotent.
	
	\ref{F is lattice homomorphism}. Since $\CF$ is increasing, it is clear that $\CF\left(\bcap_{i=1}^nI_i\right)\sub\bcap_{i=1}^n\CF(I_i)$. Now, suppose that $x\in\bcap_{i=1}^n\CF(I_i)$. Hence, for every $1 \leqslant i \leqslant n$, there exists $F_i\sub\CF$ such that $xe_{_{F_i}}\in I_i$. Obviously, $F=\bcap_{i=1}^n F_i\in\CF$ and, by Lemma \ref{Z and e elementary propoeties}\ref{e cup}, $xe_{_{(\bcap_{i=1}^nF_i)^c}}=xe_{_{\bcup_{i=1}^nF_i^c}}=x\prod_{i=1}^ne_{_{F_i}}\in I_j$ for every $1
	\leqslant j \leqslant n$. Therefore, $\bcap_{i=1}^n\CF(I_i)=\CF(\bcap_{i=1}^nI_i)$. 
	
	Now suppose that $I$ and $J$ are tow ideals of $R$. By part \ref{F is increasing}, $\CF(I) + \CF(J) \sub \CF(I+J)$. Let $x \in \CF(I+J)$, then $F \in \CF$ exists such that $xe_{_{F^c}} \in I + J \sub \CF(I) + \CF(J)$. Lemma \ref{I(F,I) properties}\ref{e_F in I} implies that $ x e_{_F} \in \CF(I) \sub \CF(I) + \CF(J) $. Hence
	\[x=x(e_{_{F^c}}+e_{_F})=xe_{_{F^c}}+xe_{_F}=ae_{_{F^c}}+be_{_{F^c}}+xe_{_F} \in \CF(I) + \CF(J) \]
	Thus the equality holds.
		 
	\ref{F preserves joins}. Suppose that $\{I_t\}_{t\in T}$ is a family of ideals of $R$. Considering $D$ as the family of  finite subsets of $T$, we can write:
	\begin{align*}
		\CF\Big(\bve_{t\in T}I_t\Big) 
			& =\CF\Big(\sum_{t\in T}I_t\Big)=\CF\Big(\bcup_{d\in D}\sum_{t\in d}I_t\Big)=\bcup_{d\in D}\CF\Big(\sum_{t\in d}I_t\Big) \\ 
			& =\bcup_{d\in D}\sum_{t\in d}\CF(I_t)=\sum_{t\in T}\CF(I_t)=\bve_{t\in T}\CF(I_t).
	\end{align*}
	
	\ref{F is nuclues}. It follows from parts \ref{F is increasing}, \ref{F is extensive}, \ref{F is idempotent} and \ref{F is lattice homomorphism}.
	
	\ref{F is frame homomorphism}. Since $R$ is arithmetical, $\CI(R)$ is distributive lattice and so, by Lemma \ref{distributive is frame}, it is a frame. Hence, by part (f), we are done.
\end{proof}

\begin{proposition} \label{I properties}
	Consider $I:\FF\to\CI(R)$ with $I(\CF)=\I(\CF,I)$. Then the mapping $I$ is increasing and preserves finite intersection, and also if $2$ is invertible, then it is a lattice homomorphism.
\end{proposition}
\begin{proof}
	We show that the mapping $I$ preserves finite intersection. Let $\CF_1,\cdots,\CF_n$ be filters on $\LA$. Clearly, 
	\[ I\Big(\bigcap_{i=1}^n\CF_i\Big) = \I\Big(\bigcap_{i=1}^n\CF_i,I\Big) \sub \bigcap_{i=1}^n\I(\CF_i,I) = \bigcap_{i=1}^nI(\CF_i).\]
	Using induction, we prove the inverse inclusion for $n=2$. Suppose that $x\in I(\CF)\cap I(\CG)=\I(\CF,I)\cap\I(\CG,I)$. Thus, there exist $A\in\CF$ and $B\in\CG$ such that $xe_{_{A^c}}, xe_{_{B^c}}\in I$. Taking $F=A\cup B$, clearly $F\in\CF\cap\CG$ and hence by part \ref{e cap} of Lemma \ref{Z and e elementary propoeties}, we can write:
	\[xe_{_{F^c}} = xe_{_{A^c\cap B^c}} = x(e_{_{A^c}}+e_{_A}e_{_{B^c}})=xe_{_{A^c}} + xe_{_A}e_{_{B^c}} \in I \] 
	Consequently, $ x\in I(\CF\cap\CG)$. 
	
	Clearly, $ I(\CF) + I(\CG) \sub I(\CF\vee\CG)$. Now, suppose that $x\in I(\CF\vee\CG)$. Since $\CF\vee\CG=\{A\cap B: A\in\CF \ , \ B\in\CG\}$, it follows that there exist $A\in\CF$ and $B\in\CG$ such that $xe_{_{(A\cap B)^c}}=xe_{_{A^c}}e_{_{B^c}}\in I$. Hence, $xe_{_{B^c}}\in \CF(I)$ and $xe_{_{A^c}}\in I(\CG)$, and therefore, by Lemma \ref{I(F,I) properties}\ref{e_F in I}, we have:
	\[2x = x (e_{_{B^c}} + e_{_B}+e_{_{A^c}} + e_{_A}) = xe_{_{B^c}} + xe_{_A} + xe_{_{A^c}} + xe_{_B} \in I(\CF) + I(\CG).\]
	Since, $2$ is invertible, it follows that $x \in I(\CF) + I(\CG)$, and thus the equality holds.
\end{proof}

For every $A\sub R$ and every $\CF\sub P(\LA)$, we define $Z(A)=\{Z(a):~a\in A\}$ and $Z^{-1}\CF=\{x\in R:~Z(x)\in\CF\}$. It is easy to see that $Z(E(I))=\{Z\sub\LA:~e_{_Z}\in I\}$ for every ideal $I$ of $R$. Clearly, for every $e\in E(R)$, we have $e\in E(I)$ if and only if $Z(e)\in Z(E(I))$.  

\begin{lemma}\label{Z(E) properties}
	Suppose that $\CF$ and $\CG$ are two filters on $A$, $\AI$ and $\BI$ are two ideals of $R$ and $\{ I_i \}_{i=1}^n$ is a finite family of ideals of $R$. Then the following hold.
	\begin{mylist}
		\item If $I\sub J$, then $Z(E(I)) \sub Z(E(J))$. \label{Z preserves sub}
		\item $Z\left(E\big(\bigcap_{i=1}^n I_i\big)\right)=\bigcap_{i=1}^nZ(E(I_i))$. \label{Z preserves cap}
		\item $Z(E(I))$ is a filter on $A$. Also, $Z(E(I))$ is a proper filter \ff $\AI$ is a proper ideal.\label{Z(E(I)) is filter}				
		\item If for each $\la\in\LA$, $\AI_\la$ is a proper ideal of $R_\la$, then $Z\left(E(\I(\CF,\prod_{\la \in\LA}\AI_\la))\right)=\CF $. \label{Z(E(product)) = F}
		\item If for each $\la \in\LA $, $\AI_\la$ is a proper ideal of $R_\la$ and $\CF$ is a proper filter, then $\I\left(\CF , \prod_{\la\in\LA} \AI_\la \right)$ is a proper ideal. \label{I(product) is proper}
	\end{mylist}
\end{lemma}
\begin{proof}
	\ref{Z preserves sub} and \ref{Z preserves cap}. They are straightforward.
	
	\ref{Z(E(I)) is filter}. $\e{X}=0\in I $, so $X\in Z(E(I))$. Let $F\in Z(E(I))$ and $F\sub H\sub A$. Thus $\e{F}\in I$ and so $\e{H}=\e{F}\e{H}\in I$. Consequently, $H\in Z(E(I))$. 
	
	Now suppose that $E,F\in Z(E(I))$. Hence, $\e{E}, \e{F}\in \AI $. Lemma \ref{Z and e elementary propoeties}\ref{e cap} concludes that 
	\[\e{E\cap F}=\e{E\cup F^c}+\e{F}=\e{E}\e{F^c}+\e{F}\in I\]
	Therefore, $E\cap F\in Z(E(I))$. To complete the proof, we can write:
	\[I=R\quad\LRTA\quad 1\in I\quad \LRTA\quad\e{\emp}\in I\quad\LRTA\quad \emp\in Z(E(I))\quad\LRTA\quad Z(E(I))=P(A)\]	
	
	\ref{Z(E(product)) = F}. Suppose that $Z \in Z\left(E(\I(\CF,\prod_{\la \in\LA} \AI_\la))\right)$, then $\e{Z}\in\I(\CF,\prod_{\la \in\LA}\AI_\la)$, so $F \in \CF$ exists such that $\e{z}\e{F^c}\in\prod_{\la \in\LA}\AI_\la$ and thus $\e{Z\cup F^c}\in\prod_{\la\in\LA}\AI_\la$, by Lemma \ref{Z and e elementary propoeties}\ref{e cup}. Since $\AI_\la$ is a proper ideal, for each $\la \in\LA$, it follows that $\e{Z\cup F^c}=0$, so $Z\cup F^c=X$, hence $F\sub Z$ and therefore $Z\in\CF$. Consequently, $Z\left(E(\I(\CF,\prod_{\la \in\LA} \AI_\la))\right) \sub \CF$. Now suppose that $F\in\CF$. By Lemma \ref{Z and e elementary propoeties}\ref{e cup}, $\e{F}\e{F^c}=\e{F \cup F^c}=\e{X}=0 \in \prod_{\la \in\LA} \AI_\la$. Hence $\e{F}\in\I(\CF,\prod_{\la \in\LA} \AI_\la )$ and thus $F\in Z\left(E(\I(\CF,\prod_{\la \in\LA}\AI_\la))\right)$. This implies that $Z\left(E(\I(\CF,\prod_{\la \in\LA} \AI_\la))\right)\supseteq \CF$ and therefore $Z\left(E(\I(\CF,\prod_{\la \in\LA} \AI_\la))\right)=\CF$.
	
	\ref{I(product) is proper}. By part \ref{Z(E(product)) = F}, $Z\left(E(\I\left(\CF,\prod_{\la \in\LA} \AI_\la)\right)\right)=\CF$ is a proper filter and thus $ \I\left(\CF,\prod_{\la \in\LA} \AI_\la\right) $ is proper ideal, by part \ref{Z(E(I)) is filter}.
\end{proof}

\begin{example} \ \\
	\begin{mylist}
		\item The inclusion relation of part \ref{I is an ideal} of Lemma \ref{I(F,I) properties} may be strict. To see this, suppose that $R_n=\BBR$ for every $n\in\BBN$ and $R=\prod_{n\in\BBN}R_n$. Assume that $I=\prod_{n\in\BBN}\{0\}$ and $\CF=\{A\sub\BBN:~1\in\LA\}$. Clearly, $\CF$ is a filter on $\BBN$ and $I=\{0\}\subsetneq\{0\}\times\prod_{n=2}^\infty\BBR=i(\CF,I)$.
		\item The converse of part \ref{F is increasing} of Proposition \ref{F properites}, in general, is not true. To see this, suppose that $R_n$ is a ring for every $n\in\BBN$, $R=\prod_{n\in\BBN}R_n$, $\CF$ is an ultrafilter on $\BBN$ such that
		$6\BBN\in\CF$, $I=<e_{2\BBN}>$, and $J=<e_{3\BBN}>$. One can see that $i(\CF,I)=R=i(\CF,J)$ while $I$ and $J$ are not comparable.
		\item By Proposition \ref{I properties}, $\CF \sub \CG$ implies that $I(\CF) \sub I(\CG)$, for each filters $\CF$ and $\CG$. The converse of this fact, in general, is not true. To see this, suppose that $R_n$ is a ring for every $n\in\BBN$, $R=\prod_{n\in\BBN}R_n$, and $\CF$ and $\CG$ are two distinct ultrafilter on $\BBN$ such that $2\BBN+1\in\CF\cap\CG$. Taking $A=2\BBN$ and $I=\Ge{e_{_A}}$, one can see that $i(\CF,I)=R=i(\CG,I)$ while $\CF$ and $\CG$ are not comparable.
	\end{mylist}
\end{example}

The following proposition shows that the primeness condition in part \ref{F sub G} of Lemma \ref{I(F,I) properties} is not necessary. 

\begin{proposition}
	Let $R=\prod_{\la\in\LA}R_\la$, $\CF$ and $\CG$ be two filter on $\LA$. Then $ \I(\CF,\{0\}) \sub \I(\CG,\{0\}) $ \ff $\CF\sub\CG$.
\end{proposition}
\begin{proof}
	$\RTA$) Suppose that $A\in\CF$. Then, we can write:
	\begin{align*}
		e_{_A}e_{_{A^c}}=0  & \RTA \quad e_{_A}\in\I(\CF,0)\sub\I(\CG,0) \\
							& \RTA \quad \ex B \in \CG  \quad e_{_A}e_{_{B^c}}=0 \\
							& \RTA \quad \ex B\in\CG \quad e_{_{A\cup B^c}}=0 \\
							& \RTA \quad \ex B\in\CG \quad A\cup B^c=X \\
							& \RTA \quad \ex B\in\CG \quad B\sub A \quad \RTA \quad A\in\CG
	\end{align*}
	Thus $\CF\sub\CG$.
	
	$\LTA$) It follows from Proposition \ref{I(F,I) properties}.
\end{proof}

\begin{lemma}\label{32.5}
	Suppose that $x , y \in R $, $I$ is an ideal of $ R $ and $Y , Z \sub A$. Then
	\begin{mylist}
		\item $Z^{-1}\CF$ is a proper ideal of $R$ for every filter $\CF$ on $\LA$. \label{Z^-1 is filter}
		\item If $x_\la^2\geq0$ for every $x_\la\in R_\la$ (for example if $R_\la$ is an $f$-ring), then $Z(I)$ is a base for a filter for every ideal $I$ in $R$ (Now this question arises: when is $Z(I)$ a filter?). \label{Z(I) is a base}
	\end{mylist}
\end{lemma}
\begin{proof}
	\ref{Z^-1 is filter}. Since $Z(0) = X \in \CF$, $0 \in Z^{-1}(\CF) \neq \emptyset$. Suppose that $x,y\in Z^{-1}(\CF)$, then $Z(x)\cap Z(y)\in\CF$. Thus, by Lemma \ref{Z and e elementary propoeties}\ref{Z cap sub},  $Z(x)\cap Z(y)\sub Z(x+y) \in \CF$. Hence $x+y\in Z^{-1}(\CF)$. Now suppose that $r\in R$ and $x\in Z^{-1}(\CF)$. Then $Z(x)\in \CF$, since $Z(x) \sub Z(rx)$, it follows that $Z(rx)\in\CF$ and consequently $rx\in Z^{-1}(\CF)$.
	
	\ref{Z(I) is a base}. Assume that  $Z(x),Z(y)\in Z(I)$ where $x,y\in I$. Clearly, $x^2+y^2\in I$ and by the hypothesis, $Z(x)\cap Z(y)=Z(x^2+y^2)\in Z(I)$.
\end{proof}

\begin{proposition}\label{maps I and F}
	Suppose that $ \mathcal{U} $ is an ultrafilter on $ \LA $ and $ \AI $ is an ideals of $ R $. Then the following hold.
	\begin{mylist}		
		\item If $ \MI_\la $ is a maximal ideal of $ R_\la $, for each $ \la \in\LA $, then $ \I(\mathcal{U},\prod_{\la \in\LA} \MI_\la) $ is a maximal ideal of $ R $. \label{I is maximal}
		
		\item If $I$ is a pseudoprime ideal, then $Z(E(I))$ is an ultrafilter. \label{F(I) is ultrafilter}
		
		\item There is some maximal ideal $M$ of $ R $ such that $ \mathcal{U}=Z(E(M)) $. \label{U=F(M)}
		
		\item If $ \MI $ is a maximal ideal, then $Z(E(M))$ is a fixed filter \ff there is some $ \beta \in\LA $ and $ \MI_\beta \in \Ma(R_\beta) $ such that $ \MI=\pi^{-1}_\beta (\MI_\beta) $. \label{F(M) is fixed}
	\end{mylist}
\end{proposition}
\begin{proof}
	\ref{I is maximal}. Lemma \ref{Z(E) properties}\ref{I(product) is proper} implies that $ \I(\mathcal{U},\prod_{\la \in\LA} \MI_\la) $ is a proper ideal. Suppose that $ x \notin \I(\mathcal{U},\prod_{\la \in\LA} \MI_\la) $. Set $ F=\{ \la \in\LA : x_\la \in \MI_\la \} $. Since $ x \notin \I(\mathcal{U},\prod_{\la \in\LA} \MI_\la) $, Lemma \ref{I(F,I) properties}\ref{x in I(product)} implies that $ F \notin \mathcal{U} $ and thus $ F^c \in\CU$. For each $\la \in F^c $, $ x_\la \notin \MI_\la $, so $ r_\la \in R_\la $ and $ m_\la \in \MI_\la $ exist such that $ 1_\la=r_\la x_\la+m_\la $. Suppose that for each $\la\in\LA$
	\[ s_\la=\begin{cases}
		r_\la & \la \notin F \\
		1_\la & \la \in F
	\end{cases} \qquad \text{ and } \qquad
	n_\la=\begin{cases}
		m_\la 	        & \la \notin F \\
		1_\la - x_\la & \la \in F
	\end{cases}. \]
	Since $ F^c \in \mathcal{U} $, Lemma \ref{I(F,I) properties}\ref{x in I(product)} implies that $ n \in \I\left(\mathcal{U},\prod_{\la \in\LA} \MI_\la\right) $. It is easy to show that $ sx+n=1 $. Consequently, $ \I\left(\mathcal{U},\prod_{\la \in\LA} \MI_\la\right) $ is a maximal ideal.
	
	\ref{F(I) is ultrafilter}. Suppose  that $ Z \notin Z(E(I))$, then $ e_{_Z} \notin \AI $. Since $ e_{_Z} e_{_{Z^c}}=0 \in \AI $ and $ \AI $ is a pseudoprime ideal, it follows that $ e_{_{Z^c}} \in \AI $ and thus $ Z^c \in Z(E(I))$. It shows that $Z(R(I))$ is an ultrafilter.
	
	\ref{F(M) is fixed}$ \Rightarrow $. Since $ \MI $ is maximal, by part \ref{F(I) is ultrafilter}, it follows that $ Z(E(M))$ is an ultrafilter, since $Z(E(M))$ is a fixed filter, there is $ \beta \in\LA $ such that $Z(E(M))=\{ Z \subseteq A : \beta \in Z \} $ and therefore $ \e{\{\beta\}} \in \MI $.  If $ 1_\beta \in \pi_\la(\MI) $, then $ x \in \MI $ exists such that $ \pi_\beta(x)=1_\beta $ and thus
	\[ 1=\e{\{\beta\}}+\e{\{\beta\}^c}=\e{\{\beta\}}+x \e{\{\beta\}^c} \in \MI \]
	which is a contradiction. This shows that $ \pi_\beta(\MI) $ is a proper ideal of $ R_\beta $, so $ \MI_\beta \in \Ma(R_\beta) $ exists such that $ \pi_\beta(\MI) \subseteq \MI_\beta $ and therefore $ \MI \subseteq \pi_\beta^{-1}(\MI_\beta) $. Since $ \MI $ is a maximal ideal, it follows that $ \MI=\pi_\beta^{-1}(\MI_\beta) $. 
	
	\ref{F(M) is fixed}$ \Leftarrow $. Suppose that $ \beta \in\LA $ and $ \MI_\beta \in \Ma(R_\beta) $, then
	\begin{align*}
		Z \in Z(E(\pi^{-1}_\beta(\MI_\beta))
		& \quad \Rightarrow \quad \e{Z} \in \pi^{-1}_\la(\MI_\beta) \quad \Rightarrow
		\quad \pi_\beta(\e{Z}) \in \MI_\beta \\
		& \quad \Rightarrow \quad \pi_\beta(\e{Z})=0 \quad \Rightarrow \quad \beta \in
		Z(\e{Z})=Z  
	\end{align*} 
	This shows that $Z(E(\pi^{-1}_\beta(\MI_\beta)) $ is a fixed filter.
\end{proof}

\begin{theorem}
	If $ \{ R_\la \}_{\la \in\LA} $ is an infinite family of rings, then $ \left| \Ma(R) \right| \geqslant 2^{2^{|\LA|}} $
\end{theorem}
\begin{proof}
	For each $ \la \in \LA $, suppose that $ \MI_\la $ is a fixed element of $ \Ma(R_\la) $ and $ \AI=\prod_{\la \in\LA} \MI_\la $. By Lemma \ref{Z(E) properties}\ref{Z preserves sub} and Proposition \ref{maps I and F}\ref{I is maximal}, $ I $ is a one to one map from the family of all ultrafilters on $ \LA $ into the family of all maximal ideal of $ R $ and thus, by \cite[Theroem 9.2]{Gillman1956Rings}, $ \left| \Ma(R) \right| \geqslant 2^{2^{|\LA|}} $.
\end{proof}

\begin{theorem}\label{cardinal of Max in product of locales}
	If $ \{ ( R_\la , \MI_\la) \}_{\la \in\LA} $ is an infinite family of local rings, then $ \left| \Ma(R) \right|=2^{2^{|\LA|}} $
\end{theorem}
\begin{proof}
	Set $ \AI=\prod_{\la \in\LA} \MI_\la $. By Lemma \ref{Z(E) properties}\ref{Z preserves sub} and and Proposition \ref{maps I and F}\ref{I is maximal}, $ \I $ is a one to one map from the family of all ultrafilters on $ \LA $ into the family of all maximal ideal of $ R $. Now suppose that $ \MI $ is a maximal ideal of $ R $, then $Z(E(M))$ is an ultrafilter, by Proposition \ref{maps I and F}\ref{F(I) is ultrafilter}. Suppose that $ x \in \MI $ and set $ F=\{\la \in\LA : x_\la \in \MI_\la \} $. If $ \la \notin F $, then $ x_\la \notin \MI_\la $, since $ R_\la $ is a local ring, it follows that $ x_\la $ is unit. Set
	\[ r_\la=\begin{cases}
		0               & \la \in F \\
		x_\la^{-1}   & \la \notin F
	\end{cases} \]
	Clearly, $ e=rx \in E(R) $ and $ Z(e)=F $ and thus $ F \in Z(E(M))\sub\mathcal{U} $. It is readily seen that $ x\e{F^c} \in \AI $ and therefore $ x \in I(\mathcal{U}) $. Consequently, $ \MI \subseteq I(\mathcal{U}) $, since $ \MI $ is a maximal ideal, it follows that $ \MI=I(\mathcal{U}) $. This shows that $ I $ is onto and thus, by \cite[Theroem 9.2]{Gillman1956Rings}, $ \left| \Ma(R) \right|=2^{2^{|\LA|}} $. 
\end{proof}

\section{The space of maximal ideals}

In this section, $ R $ denotes a unitary commutative ring. It is easy to see that $ \Ma(R) $ and $ \Ma\big( \frac{R}{\Jac(R)}\big) $ are homeomorphic. Henceforth we can assume that each ring $ R $ is semiprimitive. For each maximal ideal $ M $ we have $ h_M(\MI)=\{ \MI \} $ and therefore $ \Ma(R) $ is $ T_1 $. It is readily seen that if the intersection of a family $ \{ h_M(x) \}_{x \in S } $ is empty, then this family does not have finite intersection and thus $ \Ma(R) $ is compact. Hence $ \Ma(R) $ is $ T_2 $ \ff it is $ T_4 $. Also, \cite[Theorem 3.5]{Aliabad2021Commutative} deduces that $ \Ma(R) $ is $ T_4 $ \ff $ R $ is Gelfand.

Suppose that $ Y \subseteq \Sp(R) $, if $ \bigcap_{\PI \in  Y } \PI=\{ 0 \} $, then for each ideal $ \AI $ of $ R $ we  have
\begin{equation}
	\An(\AI)=( 0 : \AI )=\left( \bigcap_{\PI \in Y} \PI : \AI  \right)=\bigcap_{\PI \in Y} (\PI:\AI)=\bigcap_{\AI \not\subseteq \PI \in Y} \PI=k h_Y^c(I) \label{Ann(I)}
\end{equation}

It is easy to see that for each pair ideals  $ \AI $ and $ \BI $ of $ R $,

\begin{equation}\label{h_M subset h_M^c}
	h_M(\AI) \subseteq h_M^c(\BI) \text{ \ff }  \AI+\BI=R  
\end{equation}	
and 
\begin{equation}\label{h_M^c subset h_M}
	h_M^c(\AI) \subseteq h_M(\BI) \text{ \ff }  \AI\BI=\{ 0 \}
\end{equation} 

\begin{theorem}\label{Max is T4}
	$ R $ is Gelfand \ff for each pair of comaximal ideals $ (\AI,\BI) $ of $ R $ there is an ideal $ \CI $ such that $ (\AI,\BI) $ and $ \left(\An(\CI),\BI\right) $ are two pairs of comaximal ideals of $ R $.  
\end{theorem}
\begin{proof}
	\cite[Theorem 3.5]{Aliabad2021Commutative} concludes that $ \Ma(R) $ is $ T_4 $ \ff $ R $ is Gelfand. By the facts \eqref{Ann(I)}, \eqref{h_M subset h_M^c} and \eqref{h_M^c subset h_M},  $ \Ma(R) $ is $ T_4 $ \ff for all pair of ideals $ \AI $ and $ \BI $ of $ R $
	\begin{align*}
		& \quad h_M(\AI) \subseteq h_M^c(\BI) \quad \Rightarrow \quad \exists \CI \leqslant R \quad h_M(\AI) \subseteq h_M^c(\AI) \subseteq h_M^c(\CI) \subseteq \overline{h_M^c(\CI)} \subseteq h_M^c(\BI) \\
		\equiv & \quad \AI+\BI=R \quad \Rightarrow \quad \exists \CI \leqslant R \quad \AI+\CI=R \quad \text{ and } \quad h_Mkh_M^c(\CI) \subseteq h_M^c(\BI) \\
		\equiv & \quad \AI+\BI=R \quad \Rightarrow \quad \exists \CI \leqslant R \quad \AI+\CI=R \quad \text{ and } \quad h_M\left(\An(\CI)\right) \subseteq h_M^c(\BI) \\
		\equiv & \quad \AI+\BI=R \quad \Rightarrow \quad \exists \CI \leqslant R \quad \AI+\CI=R \quad \text{ and } \quad \An(\CI)+\BI=R
	\end{align*}
	It completes the proof.
\end{proof}

\begin{lemma}\label{limit and interior point}
	Suppose that $ \mathcal{A} \subseteq Y \subseteq \Ma(R) $. For each $ \MI \in Y $,
	\begin{mylist}
		\item $ \MI \in \Cl_Y \mathcal{A}  $ \ff $ \MI \in h_Yk(\mathcal{A}) $. \label{Cl in Y}
		\item $ \MI \in \In_Y \mathcal{A} $ \ff $ \MI \in h_Y(x) \subseteq h_Y^c(1-x) \subseteq \mathcal{A} $, for some $ x \in R $. \label{Int in Y}
		\item $ \MI  \in \mathcal{A}' $ \ff $ \MI $ can be omitted in $ \bigcap \mathcal{A} $. \label{Lim in Y}
	\end{mylist}
\end{lemma}
\begin{proof}
	\ref{Cl in Y} and \ref{Int in Y} $ \Leftarrow $. They are clear.
	
	\ref{Int in Y} $ \Rightarrow $. Since $ \MI \in \In_Y \mathcal{A} $, there is some $ y \in R $ such that $ \MI \in h_Y^c(y) \subseteq \mathcal{A} $. Since $ \MI $ is a maximal ideal, $ r \in R $ exists such that $ 1-ry \in M $. Set $ 1-ry=x $, then it is readily seen that $ \MI \in h_Y(x) \subseteq h_Y^c(1-x)=h_Y^c(ry) \subseteq h_Y^c(y) \subseteq \mathcal{A}  $.
	
	\ref{Lim in Y}. By part \ref{Cl in Y},
	\begin{align*}
		\MI \in \mathcal{A}' & \Leftrightarrow \quad \MI \in \mathrm{Cl}_Y \big( \mathcal{A} \setminus \{ \MI \} \big) \\
		& \Leftrightarrow \quad \MI \in h_Y k \big( \mathcal{A} \setminus \{ \MI \} \big) \\
		& \Leftrightarrow \quad \bigcap_{m \neq n \in \mathcal{A}} \NI \subseteq \MI \\
		& \Leftrightarrow \quad \bigcap_{m \neq n \in \mathcal{A}} \NI \subseteq \bigcap_{n \in \mathcal{A}} \NI
	\end{align*} 
\end{proof}

By the fact (\ref{h_M^c subset h_M}), we have $ h_M^c(x) \subseteq h_M(\An(x)) $, for every $ x \in R $. But the equality need not be  hold. For example consider $ R=\Z $, then $ h_M^c(2)=\M \setminus \big\{ \Ge{2} \big\} \neq \M=h_M(\{0\})=h_M(\An(2)) $. In the following theorem, we  show that the equality holds \ff $ R $ is a von Neumann regular ring.

\begin{proposition}
	$ R $ is a von Neumann regular ring \ff $ h_M^c(x)=h_M(\An(x)) $, for every $ x \in R $.
\end{proposition}
\begin{proof}
	($ \Rightarrow $). By this fact that $ \Ma(R)=\Mi(R) $ and \cite[Theorem 2.3]{Henriksen1965Space}, it is clear.
	
	($\Leftarrow $).  Suppose that $ \MI $ is a maximal ideal of $ R $. If $ x \in \MI $, then $ \MI \in h_M(x) $, so $ \MI \notin h_M^c(x)=h_M(\An(x)) $, thus there  is some $ y \in \An(x) \setminus \MI $, hence $ y \notin \MI $ exists such that $ xy=0 $. Consequently, $ \MI $ is a minimal prime ideal and therefore $ R $ is a von Neumann regular ring.
\end{proof}

\begin{theorem}\label{isolated point in Max}
	Suppose that $ \MI \in \Ma(R) $. The following statements are equivalent.
	\begin{mylist}
		\item $ \MI $ is an isolated point of $ \Ma(R) $. \label{Isolated}
		\item $ \MI $ can not be omitted in $ \bigcap \Ma(R) $. \label{Omit}
		\item $ \MI=\An(x) $, for some $ x \in R $. \label{M=Ann}
		\item $ \MI \in \mathcal{B}(R) $. \label{M in B(R)}
	\end{mylist}
\end{theorem}
\begin{proof}
	\ref{Isolated} $ \Leftrightarrow $ \ref{Omit}. Since $ \MI $ is an isolated point of $ \Ma(R) $, it follows that $ \MI \notin \Ma(R)' $. Now Lemma \ref{limit and interior point} concludes that  $ \MI $ can not be omitted in $ \bigcap \Ma(R) $.
	
	\ref{Omit} $ \Rightarrow $ \ref{M=Ann}. Since $ \MI $ can not be omitted in $ \bigcap \Ma(R) $, it follows that $ \bigcap_{\MI \neq \NI \in \Ma(R)} \NI \not\subseteq \MI $, so $ x \in \bigcap_{\MI \neq \NI \in \Ma(R)} \NI \setminus \MI $ exists and thus $ \An(x)=kh_M^c(x)=\MI $.
	
	\ref{M=Ann} $ \Rightarrow $ \ref{M in B(R)}. Since $ \An(x) $ is a maximal ideal, $ \An(x) $ is an affiliated prime ideal, so $ \MI \in \mathcal{B}(R) $, by \cite[Theorem 2.1]{Aliabad2013Fixedplace}.
	
	\ref{M in B(R)} $ \Rightarrow $ \ref{Isolated}. By \cite[Theorem 2.1]{Aliabad2013Fixedplace}, there is some $ x \in R $ such that $ \An(x)=\MI $, since $ \An(x)=kh_M^c(x) $, it follows that $ h_M^c(x)=\{ \MI \} $ and thus $ \MI $ is an isolated point.
\end{proof}

The following corollaries are immediate consequences of the above theorem, \cite[Corollaries 4.2 and 4.3]{Aliabad2013Fixedplace} and \cite[Theorem 4.6]{Aliabad2013Fixedplace}.

\begin{corollary}
	The following statement are equivalent.
	\begin{mylist}
		\item $ \Ma(R)=\mathcal{B}(R) $. \label{Max=B}
		\item $ \Ma(R) $ is discrete. \label{Max discrete}
		\item $ \Ma(R) $ is a fixed-place family. \label{Max is fixed-place}
	\end{mylist}
\end{corollary}

\begin{corollary}\label{connecton between isolated points}
	If $ \Ma(R) $ has an isolated point, then $ \Mi(R) $ has an isolated point.
\end{corollary}

\begin{corollary}
	If the zero ideal of a ring $ R $ is an anti fixed-place ideal, then $ \Ma(R) $ does not have any isolated point.
\end{corollary}

Given $\Ma(R)$ is $T_1$, if $\Ma(R)$ is finite, then $ \Ma(R) $ is $ T_1 $. Henceforth, we assume that $ \Ma(R) $ is infinite.

\begin{theorem}
	Suppose that $ \{ R_\la \}_{ \la \in\LA } $ is a family of rings. Then $ \Ma(R_\la) $ is homeomorphic to the closed subset $ \mathcal{M}_\la=\{ \pi_\la^{-1} (\MI_\la) : \MI_\la \in \Ma(R_\la) \} $ of $ \Ma \left( \prod_{\la \in\LA} R_\la \right) $.
\end{theorem}
\begin{proof}
	First we show that for each ideal $ \AI $ of $ \prod_{\la \in\LA} R_\la $, if for some $ \la \in\LA $ we have $ \pi_\la^{-1} (0) \subseteq \AI $, then $ \AI=\pi_\la^{-1} \left( \pi_\la( \AI ) \right) $. Clearly, $ \AI \subseteq \pi_\la^{-1} \left( \pi_\la (\AI) \right) $. Now suppose that $ x \in \pi_\la^{-1} \left( \pi_\la (\AI) \right) $, then $ \pi_\la(x) \in \pi_\la(\AI) $, so $ y \in \AI $ exists such that $ \pi_\la(x)=\pi_\la(y) $. Hence $ x \e{\{ \la \}} =y \e{\{ \la \}} \in \AI $. Since $ \e{\{ \la \}^c } \in \pi_\la^{-1}(0) \subseteq \AI$, it follows that $ x=x \e{\{\la\}}+x \e{\{ \la \}^c}=y \e{\{\la\}}+x \e{\{ \la \}^c} \in \AI $. This shows that $ x \in \AI $ and thus $ \AI= \pi_\la^{-1} \left( \pi_\la( \AI ) \right) $. Consequently, $ h_M( \pi_\la^{-1} \{ 0 \})=\M_\la $ and it is readily seen that $ \M_\la \cong \Ma ( R_\la ) $. 
\end{proof}

\begin{corollary}\label{maximal ideal of product}
	Suppose that $ \{ R_i \}_{1 \leqslant i \leqslant n} $ is a finite family of rings and $ \mathcal{M}_i=\{ \pi_i^{-1} (\MI_i) : \MI_i \in \Ma(R_i) \} $, for each $ 1 \leqslant i \leqslant n $. Then 
	\[ \Ma\left( \prod_{i=1}^n R_i \right) \cong \bigoplus_{i=1}^n \Ma(R_i) . \]
\end{corollary}
\begin{proof}
	It is clear that $ \Ma\left( \prod_{i=1}^n R_i \right)=\bigcup_{i=1}^n \mathcal{M}_i $, $ i \neq j $ implies $ \mathcal{M}_i \cap \mathcal{M}_j=\emp $ and for every $ 1 \leqslant i , j \leqslant n $, we have $ \Ma(R_i) \cong \mathcal{M}_i $ is a closed subset of $ \prod_{i=1}^n R_i $. Consequently, $ \Ma\left( \prod_{i=1}^n R_i \right) \cong \bigoplus_{i=1}^n \Ma(R_i) $.
\end{proof}

Suppose that $ \{ F_n \}_{n \in \N} $ is a family of fields. Theorem \ref{cardinal of Max in product of locales} implies that $ \left| \Ma\left( \prod_{n \in \N} F_n\right) \right|=2^{2^\N} $. On the other hands, it is readily seen that $ \N \cong \bigoplus_{n \in \N} \Ma(F_n) $, so $ \left|  \bigoplus_{n \in \N} \Ma(F_n) \right|=\omega $. Hence $ \Ma\left( \prod_{n \in \N} F_n \right) \not\cong \bigoplus_{n \in \N} F_n $. Thus the above corollary need not be true for a product of infinite family of rings. 

\begin{proposition}\label{Max is cofinte}
	$ \Ma(R) $ is a cofinite space \ff for every infinite subset $ \CF $ of $ \Ma(R) $ we have $ k( \CF)=\{0\} $.
	\label{the space is cofinite}
\end{proposition}
\begin{proof}	
	$ \Rightarrow $). Suppose that $ \CF $ is an infinite subset of $ \Ma(R) $, then $ h_Mk(\CF)=\mathrm{cl} \, \CF=\Ma(R) $, so $ k(\CF)=kh_Mk(\CF)=\bigcap \Ma(R)=\{ 0 \} $ and thus $ k(\CF)=\{0\} $.
	
	$ \Leftarrow $). Since $ \Ma(R) $ is $ T_1 $, so every finite subset of $ \Ma(R) $ is closed, so it is sufficient to show that the closure of each infinite subset of $ \Ma(R) $ is $ \Ma(R) $. Now suppose $ \CF $ is infinite, so $ \bigcap \CF=\{ 0\} $, by the assumption. Hence $ \overline{\CF}=h_Mk(\CF)=h_M(\{0\})=\Ma(R) $.	
\end{proof}

We can conclude immediately the following corollary from the above proposition.

\begin{corollary}
	If $ R $ is a Noetherian U.F.D., then $ \Ma(R) $ is a cofinite space. Thus if $ R $ is a P.I.D., then $ \Ma(R) $ is a cofinite space.
\end{corollary}

In the following example we show that the converse of Corollary \ref{connecton between isolated points} need not be true.

\begin{example}
	By Proposition \ref{the space is cofinite}, the space $ \Ma\big(\Z\big) $ is cofinite space. Clearly $ \Mi(\Z)=\big\{ \{ 0 \} \big\}  $ is discrete and $ \Ma(\Z) $ does not have any isolated point.
\end{example}

\begin{theorem}
	$ \Ma(R) $ is disconnected \ff $ R $ is direct sum of two proper ideals of the ring.
\end{theorem}
\begin{proof}
	$ \Ma(R) $ is disconnected \ff it has a nonempty clopen proper subset $ \mathcal{G} $, it is equivalent to say that there are two non-zero proper ideals $ \AI $ and $ \BI $ such that $ h_M(\AI)=h_M^c(\BI) $. Hence, by the facts \eqref{h_M subset h_M^c}  and \eqref{h_M^c subset h_M}, $ \Ma(R) $ is disconnected \ff there are two non-zero proper ideals $ \AI $ and $ \BI $ such that $ \AI+\BI=R $ and $ \AI\BI=\{ 0 \} $. It is easy to see that in this case $ \AI \cap \BI=\{ 0 \}  $ and it completes the proof.  
\end{proof}

Now, we can conclude immediately the following corollary from the above theorem and Proposition \ref{Max is cofinte}

\begin{corollary}
	If the intersection of each infinite family of maximal ideals of a ring is zero, then the ring is not direct sum of its two proper ideals.
\end{corollary}

If $ R $ has just $ n $ maximal ideals, then, since $ \Ma(R) $ is $T_1$, it follows that $ \Ma(R) $ is discrete and therefore $ C\big(\Ma(R)\big) \cong \prod_{i=1}^n \R $. Henceforth, we assume that $ \Ma(R) $ is infinite. Suppose $ \Ma(R) $ is infinite and intersection of each infinite family of maximal ideals of $ R $ is zero (for example $ R=\mathbb{Z} $), then  $ \Ma(R) $ is a cofinite topology, so every real valued continuous function on $ \Ma(R) $ is constant and therefore $ C\big(\Ma(\Z)\big) \cong \R $. Hence the converse of the above fact need not be true. 

\begin{proposition}
	Suppose that $ R_1 $, $ R_2 $, $ \ldots $ and $ R_n $ are rings. Then 
	\[C\Big(\Ma\big( \prod_{i=1}^n R_i \big)\Big) \cong \prod_{i=1}^n C\big(\Ma(R_i)\big).\]
\end{proposition}
\begin{proof}
	By Corollary \ref{maximal ideal of product} and \cite[1B.6]{Gillman1960Rings},
	\[ C\Big(\Ma\big( \prod_{i=1}^n R_i \big)\Big) \cong C \left( \bigoplus_{i=1}^n \Ma(R_i) \right) \cong \prod_{i=1}^n C\big(\Ma(R_i)\big). \qedhere \] 
\end{proof}

\begin{theorem}
	If $D$ is an integral domain, then $ C\big(\Ma(D)\big) \cong \R $.
\end{theorem}
\begin{proof}
	On contrary, suppose that $ C\big(\Ma(D)\big) \not \cong \R $, so there are distinct elements $\MI$ and $\NI$ in $\Ma(D)$ such that $f(\MI)=r<s=f(\NI)$. Set $E=f^{-1}\big[r-1,\frac{r+s}{3}\big]$ and $F=f^{-1}\big[\frac{r+s}{2},s+1\big]$. Then $ \MI \in E^\circ$, $ \NI \in F^\circ$ and $E$ and $F$ are disjoint closed sets in $\Ma(D)$, thus there are elements $x,y \in D$ and ideals $\AI$ and $\BI$ of $D$ such that $ \MI \in h_M^c(x) \subseteq h_M(\AI)=E $  and $ \NI \in h_M^c(y) \subseteq h(\BI)=F $. Fact  \eqref{h_M^c subset h_M} concludes that $ x \AI=\{ 0 \}  $, hence $ 0 \neq x \in \An(\AI) $, so $ \An(\AI) \neq \{ 0 \} $, since $D$ is an integral domain, it follows that $ \AI=\{ 0 \} $. Similarly, we can show that $ \BI=\{ 0 \} $. Consequently, $E \cap F= h_M(\AI) \cap h_M(\BI)=\Ma(D) \neq \emp $, which is a contradiction.
\end{proof}

In the following important theorem we show that in study of rings of real-valued continuous functions on the space of maximal ideals of rings, we can focus just on rings of real-valued continuous functions on a compact $T_4$ spaces.  

\begin{theorem}
	For each ring $ R $, there is a compact $ T_4 $ space $ Y $ such that 
	\[ C\big(\Ma(R)\big) \cong C\big(\Ma(C(Y))\big). \]
\end{theorem}
\begin{proof}
	In \cite[Theorem 3.9]{Gillman1960Rings} it is shown that for every Hausdorff space $ X $, there is some Tychonoff space $ Y $ and an onto continuous function $ \tau : X \rightarrow Y $ such that $ C(X) \cong C(Y) $. Analogously, we can show that this fact is also true, for each space $ X $. Hence there is a Tychonoff space $ Y $ and an onto continuous function $ \tau : \Ma(R) \rightarrow Y $ such that $ C\left( \Ma(R) \right) \cong C(Y) $. 
	Since $ \Ma(R) $ is compact and $ \tau $ is onto, it follows that $ Y $ is also compact and therefore $ Y $ is a compact $ T_4 $ space. Since $ \Ma(R) $ is a compact space, $ C(Y)=C^*(Y) \cong C(\beta Y) $. Since $ \beta Y \cong \Ma\big(C(Y)\big) $, it follows that $ C\big(\Ma(R)\big) \cong C\big(\Ma(C(Y))\big)  $.
\end{proof}

In the literature of functions rings, an ideal terms $z$-ideal if $Z(f)=Z(g)$ and $f$ imply that $g \in I$. This concept was extended in \cite{Mason1973z-ideals}. More generally, in the context of commutative rings, an ideal $I$ is $z$-ideal if $h_M(x)=h_M(y)$ and $x \in I$ imply that $y \in I$. This extension has led researchers to demonstrate that $h_M(x)$'s can serve some functions of zero-sets, for examples see  \cite{Aliabad2020Extension,Aliabad2021Commutative}.

We know that 
\begin{mylist}
	\item $Z(f)=X$ \ff $ f=0 $;
	\item $Z(f)=\emp $ \ff $f$ is unit;
	\item $Z(f)^\circ=\emp$ \ff $\An(f)=\{0\} $;
	\item If $Z(f) \subseteq Z(g)^\circ$ then $g$ is a multiple of $f$. 
\end{mylist}
In the following lemma we show that $h_M(x)$'s can serve these functions.

\begin{lemma} \label{Elementary h_M}
	Suppose that $ x , y \in R $. Then 
	\begin{mylist}
		\item $ h_M(x)=\Ma(R) $ \ff $ x=0 $. \label{h_M=Max}
		\item $ h_M(x)=\emp $ \ff $ x $ is unit. \label{h_M=empty}
		\item $ h_M(x)^\circ=\emp $ \ff $ \An(x)=\{ 0 \} $. \label{h_M^o=empty}
		\item $ h_M(x) \subseteq h_M(y)^\circ $ implies that $ y $ is a multiple of $ x $. \label{multiple}   
	\end{mylist}
\end{lemma}
\begin{proof}
	\ref{h_M=Max} and \ref{h_M=empty}. They are evident.
	
	\ref{h_M^o=empty}. By \cite[Lemma 3.7]{Aliabad2020Extension}, 
	\begin{align*}
		h_M(x)^\circ=\emp & \quad \Leftrightarrow \quad h_M^c(\An(x))=\emp \\
		& \quad \Leftrightarrow \quad h_M(\An(x))=\Ma(R) \\
		& \quad \Leftrightarrow \quad \An(x)=\{ 0 \}
	\end{align*}
	\ref{multiple}. By \cite[Lemma 3.7]{Aliabad2020Extension}, $ h_M(x) \subseteq h_M^c(\An(y)) $ and thus $ \langle x  \rangle+\An(y)=R $, by the fact (\ref{h_M subset h_M^c}). Hence $ z \in \An(y) $ and $ r \in R $ exist such that $ rx+z=1 $ and therefore 
	\[ rxy=rxy+0=rxy+zy=( rx+z ) y=y \] 
\end{proof}

Now in the following theorem we improve \cite[Exercises \S 10.19 and \S 4.15]{Lam1991First} and extend \cite[Theorem ]{Gillman1960Rings} and \cite[Theorem 5.8]{Aliabad2021Commutative} for primitive rings.

\begin{theorem}
	The following statements are equivalent.   
	\begin{mylist}
		\item $h_M(x)$ is open, for each $ x \in R $. \label{h_M is open}
		\item Every ideal of $R$ is a $z$-ideal. \label{Every is z-ideal}
		\item Every ideal of $R$ is a strong $z$-ideal. \label{Every is srong}
		\item Every ideal of $R$ is Hilbert. \label{Every is Hilbert}
		\item Every finitely generated ideal of $R$ is a $z$-ideal. \label{Every finitely is z-ideal}
		\item Every finitely generated ideal of $R$ is a strong $z$-ideal. \label{Every finitely is strong}
		\item Every principal ideal of $R$ is a $z$-ideal. \label{Every principal is z-ideal}
		\item Every principal ideal of $R$ is a strong $z$-ideal. \label{Every principal is strong}
		\item Every essential ideal of $R$ is a $z$-ideal. \label{Every essential is z-ideal}
		\item Every essential ideal of $R$ is a strong $z$-ideal. \label{Every essential is strong}
		\item Every ideal of $R$  is semiprime. \label{Every is semiprime}
		\item $R$ is regular. \label{R is regular}
		\item Every prime ideal of $R$ is maximal. \label{Every prime is maximal}
	\end{mylist}
\end{theorem}
\begin{proof}
	\ref{h_M is open} $ \Rightarrow $ \ref{Every is z-ideal}. Suppose that $ \AI $ is an ideal of $ R $, $ h_M(x)=h_M(y) $ and $ x \in \AI $. Since $ h_M(y) $ is open, $ h_M(x)=h_M(y)=h_M(y)^\circ $, thus $ y $ is a multiple of $ x $, by Lemma \ref{Elementary h_M}\ref{multiple} and therefore $ y \in \AI $. This Shows that $ \AI $ is a $z$-ideal.
	
	\ref{Every is z-ideal} $ \Leftrightarrow $ \ref{Every is srong} $ \Leftrightarrow $ \ref{Every is Hilbert} $ \Leftrightarrow $ \ref{Every finitely is z-ideal} $ \Leftrightarrow $ \ref{Every finitely is strong} $ \Leftrightarrow $ \ref{Every principal is z-ideal} $ \Leftrightarrow $ \ref{Every principal is strong} $ \Leftrightarrow $ \ref{Every essential is z-ideal} $ \Leftrightarrow $ \ref{Every essential is strong} $ \Leftrightarrow $ \ref{R is regular}. By the assumption $Y=\Ma(R) $, they follow from \cite[Theorem 5.8]{Aliabad2020Extension}.
	
	\ref{Every is semiprime} $ \Leftrightarrow $ \ref{R is regular}. By the assumption $ Y=\Sp(R) $, it follows from \cite[Theorem 5.8]{Aliabad2020Extension}.
	
	\ref{R is regular} $ \Rightarrow $ \ref{h_M is open}. Suppose that $ x \in R $. Since $ R $ is regular, $ r $ exists such that $ x^2 r=x $. Set $ y=1 - xr $, then 
	\[ \begin{cases}
		y+xr=1 \\
		yx=(1-rx)=x - xr=0 
	\end{cases} \quad \Rightarrow \quad \An(x)+\langle x \rangle=R \]
	So, by the fact (\ref{h_M subset h_M^c}), we deduce that $ h_M(x) \subseteq h_M^c(\An(x)) $ ($\star$). Since $ x \An(x)=0 $, the fact (\ref{h_M^c subset h_M}) follows that $ h_M^c(\An(x)) \subseteq h_M(x) $ $(\star \star)$.
	
	By $ (\star) $ and $ (\star \star) $, we have $ h_M(x)=h_M^c(\An(x)) $ and therefore $ h_M(x) $ is open.
	
	\ref{R is regular} $\Leftrightarrow$ \ref{Every prime is maximal}. It follows from \cite[Excersise \S 10.19]{Lam1991First}.   
\end{proof}

We see in \cite{Kim2003Almost,Levy1977Almost} that for each topological $X$ the following are equivalent
\begin{mylist}
	\item $X$ is almost $P$-space.
	\item Every non-empty zero-set has non-empty interior.
	\item Every regular element in $C(X)$ has the inverse element.
	\item $C(X)$ has no proper regular ideal.
\end{mylist}
In the following corollary we show that $h_M(x)$'s can give the role of zero-sets in the above mentioned statement.

\begin{corollary}
	The following statements are equivalent.
	\begin{mylist}
		\item $ h_M(x) \neq \emp $ implies that $ h_M(x)^\circ \neq \emp $, for each $ x \in R $. \label{almost P-space}
		\item Every regular element of $ R $ is unit. \label{Every regular is unit}
		\item $ R $ has no regular proper ideal. \label{Has no regular proper}
	\end{mylist}
\end{corollary}
\begin{proof}
	\ref{almost P-space} $\Leftrightarrow$ \ref{Every regular is unit}. It follows from parts \ref{h_M=empty} and \ref{h_M^o=empty} of Lemma \ref{Elementary h_M}.
	
	\ref{Has no regular proper}. $\Leftrightarrow$ \ref{Every regular is unit}. It is straightforward.
\end{proof}


\begin{thebibliography}{99}
	\bibitem{Abdollahpour2024Homorphisms} E. Abdollahpour, A. R. Aliabad and J. Hashemi, \emph{On the homomorphisms of $\cap$-structure spaces}, Categ. Gen. Algebr. Struct. Appl. (2024).
	\bibitem{Aliabad2013Fixedplace} A. R. Aliabad and M. Badie, \emph{Fixed-place ideals in commutative rings}, Comment. Math. Univ. Carolin. 54. 1 (2013), 53--68.
	\bibitem{Aliabad2018Bourbaki} A. R. Aliabad and M. Badie, \emph{On Bourbaki associated prime divisors of an ideal}, Quaest. Math. (2018), 1--22.
	\bibitem{Aliabad2020Extension} A. R. Aliabad, M. Badie and S. Nazari, \emph{An Extension of $z$-ideals and $z^\circ$-ideals}, Hact. J. Math. Sta. 49(1) (2020).
	\bibitem{Aliabad2021Commutative} A. R. Aliabad, M. Badie and S. Nazari, \emph{On commutative Gelfand rings}, J. Math. Ext. 16. (2021).
	\bibitem{Atiyah1969Introduction} M. F. Atiyah and I. G. Macdonald, \emph{Introduction to commutative algebra}. Reading, 1969.
	\bibitem{Anderson2008Ideals} D. D. Anderson and J. Kintzinger, \emph{Ideals in direct products of commutative rings}, Bull. Austral. Math. Soc. 77 (2008) 477--483.
	\bibitem{Blyth2005Lattices} T. S. Blyth, \emph{Lattices and Ordered Algebraic Structures}, Springer, 2005.
	\bibitem{Davey2002Introduction} B. A. Davey and H. A. Priestly, \emph{Introduction to lattices and order}, Cambridge University Press, 2002.
	\bibitem{Dikranjan2013Categorical} D. Dikranjan and W. Tholen, \emph{Categorical structure of closure operators: with applications to topology, algebra and discrete mathematics}, Springer Science \& Business Media 346, 2013.	
	\bibitem{Dow1988Space} A. Dow, M. Henriksen, R. Kopperman and J. Vermeer \emph{The space of minimal prime ideals of $C(X)$ need not be basically disconnected}, Proc. Amer. Math. Soc. 104(1) (1988) 317-320.
	\bibitem{Finocchiaro2023Prime} C. A. Finocchiaro, S. Frisch and D. Windisch, \emph{Prime ideals in infinite product of commutative rings}, Commun. Contemp. Math. 26(08) (2024) 2350045.
	\bibitem{Gillman1957Rings} L. Gillman, \emph{Rings with Hausdorff structure space}, Fund. Math. 45 (1957) 1--16.
	\bibitem{Gillman1956Rings} L. Gillman and M. Jerison, \emph{Rings of continuous functions}. The University Series in Higher Mathematics, Princeton-Toronto-London-New York, D. Van Nostrand Company, Inc. IX, 300 p. 1960.
	\bibitem{Gillman1960Rings} L. Gillman and M. Henriksen, \emph{Rings of continuous functions in which every finitely generated ideal is principal}, Trans. Amer. Math. Soc. 82, 2 (1956), 366--391.
	\bibitem{Gilmer1992Product} R. Gilmer and W. Heinzer, \emph{Products of commutative rings and zero-dimensionality}, Trans. Amer. Math. Soc. 331(2) (1992) 663--680.
	\bibitem{Gilmer1989imbedding} R. Gilmer and W. Heinzer, \emph{On the imbedding of a direct product into a zero-dimensional commutative ring}, Proc. Amer. Math. Soc. 106 (1989) 631--637.
	\bibitem{Gilmer1995Infinite} R. Gilmer and W. Heinzer, \emph{Infinite products of zero-dimensional commutative rings}, Houston J. Math. 21(2) (1995) 247--259.
	\bibitem{Gratzer2011Lattice} G. Gr\"atzer, \emph{Lattice Theory}, Foundation, Springer Basel AG, 2011.
	\bibitem{Hashemi2020Cap} J. Hashemi, \emph{On the $\cap$-structure spaces}, J. Math. Extension, 14(3)  (2020), 225--236.
	\bibitem{Henriksen1965Space} M. Henriksen  and M. Jerison, \emph{The space of minimal prime ideals of a commutative ring}. Trans. Amer. Math. Soc. 115 (1965), 110--130.
	\bibitem{Kim2003Almost} C.I. Kim, \emph{Almost P-spaces}, Commun. Korean Math. Soc. 18(4) (2003) 695--701.
	\bibitem{Kohls1958space} C. Kohls, \emph{The space of prime ideals of a ring}, Fund. Math. 45(1) (1958) 17--27.
	\bibitem{Lam1991First} T. Y. Lam, \emph{A first course in noncommutative rings}, Graduate Texts in Mathematics/Springer-Verlag 131, 1991.
	\bibitem{Levy1977Almost} R. Levy, \emph{Almost-P-spaces}, Canad. J. Math. 29(2) (1997) 284--288.
	\bibitem{Levy1991Prime} R. Levy, P. Loustaunau and J. Shapiro, \emph{The prime spectrum of an infinite product of copies of $\mathbb{Z}$}, Fund. Math. 138 (1991) 155--164
	\bibitem{Mason1973z-ideals} G. Mason, \emph{$z$-ideals and Prime ideals}, J. Algebra 26 (1973) 280--297.
	\bibitem{Olberding2005Prime} B. Olberding and J. Shapiro, \emph{Prime ideals in ultraproducts of commutative rings}, J. Algebra 285(2) (2005) 768--794.
	\bibitem{Sharp2000Steps} R. Sharp, \emph{Steps in commutative algebra}, Cambridge university press, 2000.
	\bibitem{Tarizadeh2020Projectivity} A. Tarizadeh, \emph{On the projectivity of finitely generated flat modules}, Extracta Math. 35(1) (2020) 55--67.
	\bibitem{Tarizadeh2023Tame} A. Tarizadeh and N. Shirmohammadi, \emph{Tame and wild primes in direct product of commutative rings}, Arxive, (2023).
	\bibitem{Walker1974Stone} R. Walker, \emph{The Stone-\v Cech compactification}, Springer-Verlag, Berlin and New York, 1974.
	\bibitem{Willard1970General} S. Willard, \emph{General topology}, Addison-Wesley Series, New York, 1970.
\end{thebibliography}
\end{document}